\documentclass[11pt]{amsart}
\usepackage{amsmath, amsthm, amssymb, amsfonts,color,url,enumitem, caption, booktabs}
\usepackage[toc,page]{appendix}
\usepackage{cancel}
\usepackage{wrapfig,hyperref}
\usepackage{mathtools}
\usepackage{fullpage}
\usepackage{subfig}
\usepackage{color}
\usepackage[normalem]{ulem}
\usepackage[dvipsnames]{xcolor}

\newtheorem{definition}{Definition}[section]
\newtheorem{theorem}{Theorem}[section]

\newtheorem{lemma}{Lemma}[section]

\newtheorem{proposition}{Proposition}[section]

\newcommand{\Z}{\mathbb{Z}}
\newcommand{\Q}{\mathbb{Q}}
\newcommand{\R}{\mathbb{R}}

\newcounter{main-constants}
\newcommand{\newconstant}[1]{\refstepcounter{main-constants}\label{#1}}
\newcounter{lemma-constants}
\newcommand{\newconstantlem}[1]{\refstepcounter{lemma-constants}\label{#1}}

\title{On the Diophantine equations of the form $\lambda_1U_{n_1} + \lambda_2U_{n_2} +\ldots + \lambda_kU_{n_k} = wp_1^{z_1}p_2^{z_2} \cdots p_s^{z_s}$}
\author{Eva Goedhart, Brian Ha, Lily McBeath, and Luisa Velasco}
\date{\today}
\thanks{This work was partially supported by NSF Grant DMS1947438 and Williams College through the SMALL Research Experience for Undergraduates.}

\begin{document}
\maketitle

\begin{abstract}
    In this paper, we consider the Diophantine equation $\lambda_1U_{n_1}+\ldots+\lambda_kU_{n_k}=wp_1^{z_1} \cdots p_s^{z_s},$ where $\{U_n\}_{n\geq 0}$ is a fixed non-degenerate linear recurrence sequence of order greater than or equal to 2; $w$ is a fixed non-zero integer; $p_1,\dots,p_s$ are fixed, distinct prime numbers; $\lambda_1,\dots,\lambda_k$ are strictly positive integers; and $n_1,\dots,n_k,z_1,\dots,z_s$ are non-negative integer unknowns. We prove the existence of an effectively computable upper-bound on the solutions $(n_1,\dots,n_k,z_1,\dots,z_s)$. 
    In our proof, we use lower bounds for linear forms in logarithms, extending the work of Pink and Ziegler (2016), Mazumdar and Rout (2019), Meher and Rout (2017), and Ziegler (2019).
\end{abstract}
\section{Introduction}
There has been recent interest in bounding solutions of Diophantine equations consisting of terms of recurrence sequences set equal to a prime power. For instance, in 2014, Bravo and Luca \cite{bravo2016diophantine} resolved the following equation for all indices $n$ and $m$ and exponent $a$,
\[
    F_n + F_m = 2^a,
\]
where~$F_i$ is the $i$-th term in the Fibonacci sequence. Further, results of this type have been extended to involve terms of any sufficiently nice binary recurrence sequence. In 2018, Pink and Ziegler \cite{pink2016effective} effectively bounded the solutions to the following equation for fixed primes $p_1, \dots, p_s$ and nonzero integer constant $w$ and variables $n,m,z_1,\dots,z_s$,
\[
    u_n + u_m = wp_1^{z_1} \cdots p_s^{z_s},
\]
where~$u_n$ and~$u_m$ are binary recurrence terms. Moreover, in 2019, Mazumdar and Rout \cite{mazumdar2019prime} studied the equation $$u_{n_1} +\ldots + u_{n_k} = p^z,$$ where they extended the number of terms to an arbitrary finite sum equaling a prime power and achieved an effective finiteness result.

Additionally, in 2019, Ziegler \cite{ziegler2019effective} found effective finiteness results for the equation $$a_1U_{n_1}+\ldots+a_kU_{n_k}=b_1V_{m_1}+\ldots+b_\ell V_{m_\ell},$$ involving a more general class of sequences $\{U_n\}_{n\geq 0}$ and $\{V_m\}_{m\geq 0}$ with order greater than or equal to 2.

In this paper, we aim to generalize the results of Pink and Ziegler \cite{pink2016effective}, Mazumdar and Rout \cite{mazumdar2019prime}, Meher and Rout \cite{meher2017linear}, and Ziegler \cite{ziegler2019effective}, and analyze the integer solutions to the more general Diophantine equation given by 
\begin{align}\label{eq:main-equation-generalized}
\lambda_1U_{n_1} + \lambda_2U_{n_2} +\ldots + \lambda_kU_{n_k} &= wp_1^{z_1}p_2^{z_2} \cdots p_s^{z_s},
\end{align}
where $n_1, \dots, n_k, z_1, \dots, z_s$ are non-negative integers; $\{U_n\}_{n \geq 0}$ is a non-degenerate integer recurrence sequence of order $d\geq 2$; $\lambda_1,\dots,\lambda_k$ are nonzero integers; $p_1, \dots, p_s$ are distinct primes; and $w$ is a non-zero integer with $p_i \nmid w$ for $1 \leq i \leq s$.

In particular, the main result of this paper relies on what it means for our coefficients $\lambda_1,\dots,\lambda_k$ to \textit{admit dominance}, a condition defined by Ziegler \cite{ziegler2019effective}, which is as follows.
\begin{definition}[Ziegler, \cite{ziegler2019effective}]\label{def:admit-dominance}
For a given recurrence sequence $\{U_n\}_{n\geq 0}$, the $k$-tuple of non-zero integers $(\lambda_1,\dots,\lambda_k)$ \emph{admits dominance} if for every $n_1>\ldots>n_k\geq 0$ we have \newconstant{ziegler-const} $$|\lambda_1U_{n_1}+\ldots+\lambda_kU_{n_k}|>C |U_{n_1}|,$$ where $C$ is a positive, effectively computable constant that does not depend on $n_1,\dots,n_k$.
\end{definition}

Using the admits dominance condition, we now state the main result of this paper, the other terms of which are defined in the next section.
\begin{theorem}\label{thm:main-thm-generalized}
Let $\{U_n\}_{n\geq0}$ be a non-constant, simple, non-degenerate, linear recurrence sequence defined over the integers with a dominant root $\alpha$. Assume that the $k$-tuple of nonzero integers $(\lambda_1, \dots, \lambda_k)$ admits dominance for $\{U_n\}_{n\geq0}$. Then, there exists an effectively computable constant $\mathcal{C}$ such that every solution $(n_1, \dots, n_k, z_1, \dots, z_s)$ to equation \eqref{eq:main-equation-generalized} with $n_1 > \dots > n_k$ and $n_1 \geq 3$ satisfies
$$\max\{n_1, \dots, n_k, z_1, \dots, z_s\} \leq \mathcal{C}.$$
\end{theorem}
Since $n_i,z_j$ are nonnegative integers, by bounding the solutions to equation \eqref{eq:main-equation-generalized} from above, we prove that there are finitely many solutions. 

In Section \ref{sec:background}, we provide definitions of the main terms needed in Theorem~\ref{thm:main-thm-generalized} and supply some known results that we will need in the proof. Section \ref{sec:main-theorem-proof} is dedicated to proving Theorem \ref{thm:main-thm-generalized}. We begin with proving some auxiliary lemmas. Then we divide the bulk of the proof into two distinct steps. In Section \ref{sec:generalized-bounding-difference}, we use an induction argument on $m$ where $2\leq m \leq k$ to first bound the range of the differences of the exponents $n_1-n_m$. In Section \ref{sec:generalized-bounding-n1}, using the results from the previous section and bounds on $z_i$, we attain an absolute upper bound on $n_1$. Lastly, in Table \ref{tab:constants} of the Appendix,  
we explicitly define most of the constants that are used throughout the paper. 

\section{Background}\label{sec:background}

Before we delve into the proof of our main result, we want to outline all the terms, definitions, and preliminary results that are used in the proof of our main result. To begin, let us recall the definition of a linear recurrence and what it means for it to be simple and non-degenerate.

\begin{definition}
A sequence $\{U_n\}_{n \geq 0}$ is a \emph{linear recurrence sequence over the integers} if for all integers $n\geq 0$, the $n$-th term in the sequence can be determined from the previous $d$ terms for some fixed positive integer $d$; that is, there exists an integer $d\geq 2$ and $a_i\in\Z$ for $1\leq i\leq d$ such that
$$U_n = a_1U_{n-1} + a_2U_{n-2} +\ldots + a_dU_{n-d}$$
with $U_{n-i}\in\mathbb{Z}$ for all $i\geq 1$.
\end{definition}

Note that the first $d$ terms of the sequence $\{U_n\}_{n \geq 0}$ must be given to fully determine a linear recurrence sequence, thus $d$ is called the order of $\{U_n\}$. The other main determining components of such a sequence are the coefficients $a_i$ for $1\leq i\leq d$. Define a {\it companion polynomial} to $\{U_n\}_{n\geq 0}$, which is given by $$f(x)=x^d-a_1x^{d-1}-\ldots -a_{d-1}x -a_d,$$ with roots $\alpha_1,\alpha_2,\dots,\alpha_d\in\mathbb{C}$ and degree $d$, given by the order of the sequence. Without loss of generality, we can relabel the roots so that $|\alpha_1|\geq|\alpha_2|\geq \cdots \geq |\alpha_d|$. If there is a largest root, $|\alpha_1|>|\alpha_i|$ for all $2\leq i\leq d$, then we will write the largest root of the companion polynomial $f(x)$ simply as $\alpha$, and call it the {\it dominant root}. The sequence is called \emph{non-degenerate} if, for all $1\leq i,j\leq d$ such that $\alpha_i\neq \alpha_j$, we have that $\alpha_i/\alpha_j$ is not a root of unity. We note that a non-degenerate sequence with dominant root $\alpha$ has the property $\alpha>1$; otherwise, all roots of $f(x)$ are roots of unity by a result originally due to Kronecker \cite{greiter1978simple}, thus contradicting the definition of a non-degenerate recurrence sequence. Finally, the sequence $\{U_n\}$ is said to be \emph{simple} if there exist algebraic numbers $u,u_2,\dots,u_d$, each of degree at most $d$ contained in $\Q(\alpha_1,\dots,\alpha_d)$, such that \begin{align}\label{eq:closed-form}
    U_n=u\alpha^n+\sum_{j=2}^d u_j\alpha_j^n.
\end{align}

We now provide some auxiliary results that will be used in the proof of Theorem \ref{thm:main-thm-generalized}. 

\begin{lemma}[Peth\H o, de Weger \cite{petho1986}]\label{lem:not_log}
Let $u,v,$ and $h$ be real numbers, $u,v\geq 0$, $h\geq 1$, and $x_0\in\R$ be the largest solution of the equation $$x_0=u+v(\log x_0)^h.$$ Then, $$x_0<\max\{2^h(u^{1/h}+v^{1/h}\log (h^hv))^h,2^h(u^{1/h}+2e^2)^h\}.$$
\end{lemma}

Next, for ease in notation, we follow Pink and Ziegler \cite{pink2016effective} and use $\log_*(x)$ in place of $\max\{0, \log x\}$ for $x\in \R_{>0}$. 
Specifically, for a real number $x>0$ we define  \begin{equation} \label{definition-log-asterisk}
    \log_*:\R_{>0}\to \R_{\geq 0} \text{ where }  x\mapsto \max\{0,\log x\}.
\end{equation}

\begin{definition}[Smart, \cite{nigel1998smart}]\label{def:log_height}
Let $\eta$ be an algebraic number of degree $n$ with minimal polynomial $$p(x)=q_0x^n+q_1x^{n-1}+\ldots+q_n=q_0\prod_{i=1}^n(x-\eta_i),$$ where all $q_i$ are relatively prime integers, $q_0>0$, and the $\eta_i$ are conjugates of $\eta$. Then, we define the \emph{absolute logarithmic height of $\eta$} as $$h(\eta)=\frac{1}{n}\left(\log |q_0|+\sum_{i=1}^n\log_{*}|\eta_i|\right). $$ 
Three key properties of the absolute logarithmic height are as follows, for $\eta_1,\dots, \eta_t$ algebraic numbers, $a \in \Z$.
\begin{enumerate}
    \item $h(\eta_1\cdots\eta_t) \leq \sum_{i=1}^{t}h(\eta_i)$,
    \item $h(\eta_1 +\ldots + \eta_t) \leq \log(t) + \sum_{i=1}^{t}h(\eta_i)$, and
    \item $h(\eta^a) = |a|h(\eta)$.
\end{enumerate}
\end{definition}

The following is an implication of Matveev's~\cite{matveev2000explicit} monumental result on bounding linear forms in logarithms given by Bugeaud, Mignotte, and Siksek~\cite{bugeaud2006classical} to solve exponential Diophantine equations. We will use their version here for linear recurrences. 
\begin{theorem}[Bugeaud, Mignotte, Siksek \cite{bugeaud2006classical}]\label{thm:matveev} Let $\mathbb{K}/\mathbb{Q}$ be a number field of degree $D$, let $\gamma_1,\dots,\gamma_t$ be positive real numbers in $\mathbb{K}$, and $b_1,\dots,b_t$ be rational integers. Put $$ \Lambda = \gamma_1^{b_1}\cdots\gamma_t^{b_t} - 1\quad \text{and}\quad B \geq \max\{|b_1|,\dots,|b_t|\}.$$ Let $A_i \geq \max\{Dh(\gamma_i), |\log \gamma_i|, 0.16\}$ be real numbers for $i = 1,2,\dots,t$. Then, assuming that $\Lambda \neq 0$, {we have} $$\log |\Lambda| > -1.4 \cdot 30^{t + 3} t^{4.5} D^2(1 + \log D)(1 + \log B)A_1\cdots A_t.$$
\end{theorem}

\section{Proof of Main Theorem}\label{sec:main-theorem-proof}
To begin the proof of Theorem \ref{thm:main-thm-generalized}, let $\{U_n\}_{n\geq0}$ be a non-constant, simple, non-degenerate, linear recurrence sequence defined over the integers with dominant root $\alpha$. 
Further, assume that the $k$-tuple of nonzero integers $(\lambda_1, \dots, \lambda_k)$ admits dominance for the sequence $\{U_n\}_{n\geq0}$ and assume that $(n_1,\dots,n_k,z_1,\dots,z_k)$ is a solution to equation \eqref{eq:main-equation-generalized} with $n_1 \geq 3$ and $n_1 > n_2 > \dots > n_k$.

Applying Definition \ref{def:admit-dominance} to the right hand side of our equation, we obtain $C_1$ such that \[|\lambda_1 U_{n_1} + \dots + \lambda_kU_{n_k}| > C_1|U_{n_1}|.\] In our argument, we wish to apply Matveev's theorem to obtain bounds on a linear form in logarithms. To do this, we start by applying Proposition \ref{prop:proposition-1} to the $k$-tuple $(\lambda_1,\dots,\lambda_k)$. We include the statement for convenience.
\begin{proposition}[Ziegler, \cite{ziegler2019effective}]\label{prop:proposition-1}

If $(\lambda_1,\dots,\lambda_k)$ admits dominance for $\{U_n\}_{n\geq0}$, there exists a positive, effectively computable constant \newconstant{ziegler-const-2}$C_2$ such that \begin{align}\label{eq:ziegler-prop}
    |\lambda_1\alpha^{n_1}+\ldots+\lambda_K\alpha^{n_K}|>C_2|\alpha|^{n_1}
\end{align}
for any $1\leq K\leq k$ and any integers $n_1>\ldots>n_K\geq 0$.
\end{proposition}

The computation of the constant $C_{\ref*{ziegler-const-2}}$ is detailed in \cite{ziegler2019effective} in Claim 1 which is part of the proof of Proposition 1. $C_{\ref*{ziegler-const-2}}$, which depends on $\{U_n\}$, can be found explicitly through a recursive process. We refer the interested reader to \cite{ziegler2019effective} for the details.

Moving forward, we will use a superscript in parentheses to emphasize a constant's dependence on the number of recurrence terms in our equation.

The following lemma allows us to compute the modified height used in Matveev's theorem in Section \ref{sec:main-theorem-proof}. Recall that $\{U_n\}$ has a companion polynomial with roots $\alpha_i$ and since $\{U_n\}$ is a non-degenerate recurrence sequence, $\alpha>1$ is the dominant root.
\begin{lemma}\label{lem:mazumdar-lemma3.7}
Let $D = [\Q(\alpha, \dots, \alpha_d) : \Q]$ and let
\begin{align*}
\gamma_m =   \begin{cases}
                |w| |u|^{-1}\left\lvert \lambda_1 + \lambda_2\alpha^{n_2 - n_1} +\ldots + \lambda_{m}\alpha^{n_{m} - n_1} \right\rvert^{-1},    &m \geq 2 \\
                |w| |u|^{-1}\left\lvert \lambda_1 \right\rvert^{-1},    &m = 1
                \end{cases}
\end{align*}
be an algebraic number in $\Q(\alpha, \dots, \alpha_d)$ for $1 \leq m \leq k$, where $u$ is given by equation \eqref{eq:closed-form}. Then, there exists a constant \newconstant{maz-lemma-const-1}$C_{\ref*{maz-lemma-const-1}}^{(m)}$ depending on $m, w, u$, and  $\lambda_1,\dots,\lambda_m$, such that
\begin{align*}
A_3(m)  
    &=      \begin{cases}
            D C_{\ref*{maz-lemma-const-1}}^{(m)} + D\left( \sum_{j=2}^{m} |n_1 - n_j| \right) h(|\alpha|),  &m\geq 2 \\
            D C_{\ref*{maz-lemma-const-1}}^{(m)},                                                                 &m=1
            \end{cases} \\
    &\geq   \max\{Dh(\gamma_m), |\log \gamma_m|, 0.16\}.
\end{align*}
\end{lemma}
\begin{proof}
First, we estimate $\log \gamma_m$ for $m\geq 2$. By equation \eqref{eq:ziegler-prop}, 
\begin{align*}
\gamma_m    
        &=      \frac{|w|}{|u|\left\lvert \lambda_1 + \lambda_2\alpha^{n_2 - n_1} +\ldots + \lambda_{m}\alpha^{n_{m} - n_1} \right\rvert} \\
        &\leq   \frac{|w|}{|u|C_{\ref*{ziegler-const-2}}^{(m)}}.
\end{align*}
Similarly,
\begin{align*}
\gamma_m^{-1}
        &=      |w|^{-1} |u|\left\lvert \lambda_1 + \lambda_2\alpha^{n_2 - n_1} +\ldots + \lambda_{m}\alpha^{n_{m} - n_1} \right\rvert \\
        &\leq   |w|^{-1} |u|  \left( \max_{1 \leq j \leq m}|\lambda_j| \right) m.
\end{align*}
Thus,
\begin{align} \label{eq:log-gamma3t}
|\log \gamma_m|  
    &\leq   \max\left\{\log |w| - \log |u| - \log C_{\ref*{ziegler-const-2}}^{(m)}\right.,\\
    &\quad\quad \left.\log |u| + \log \left( \max_{1 \leq j \leq m}|\lambda_j| \right) + \log m - \log |w|\right\}\nonumber.
\end{align}
Next, we estimate $h(\gamma_m)$ using the properties of the logarithmic height given in Definition \ref{def:log_height}. 
\begin{align}\label{eq:height-gamma3t}
h(\gamma_m)
    &=  h\left( |w| |u|^{-1} \left\lvert \lambda_1 + \lambda_2\alpha^{n_2 - n_1} +\ldots + \lambda_{m}\alpha^{n_{m} - n_1} \right\rvert^{-1} \right) \nonumber \\
    &\leq   \log |w| + h(|u|) + h(\lambda_1 + \lambda_2\alpha^{n_2 - n_1} +\ldots + \lambda_{m} \alpha^{n_{m} - n_1}) \nonumber \\
    &\leq   \log |w| + h(|u|) + \left( \sum_{j=1}^{m} h(\lambda_j) \right) + \left( \sum_{j=2}^{m} |n_j - n_1| \right)  h(|\alpha|) + \log m \nonumber \\
    &\leq   \log |w| + h(|u|) + \left( \sum_{j=1}^{m} \log |\lambda_j| \right) + \left( \sum_{j=2}^{m} |n_j - n_1| \right)  h(|\alpha|) + \log m \nonumber \\
    &\leq   \log |w| + h(|u|) + m  \log \left( \max_{1 \leq j \leq m}|\lambda_j| \right) + \left( \sum_{j=2}^{m} |n_j - n_1| \right)  h(|\alpha|) + \log m.
\end{align}
Comparing inequalities \eqref{eq:log-gamma3t} and \eqref{eq:height-gamma3t}, and multiplying by $D$ gives our desired result, where we define
\begin{align*}
    C_{\ref*{maz-lemma-const-1}}^{(m)} &= \log |w| + h(|u|) + m \log \left( \max_{1 \leq j \leq m}|\lambda_j| \right) + \log m \\
    &\quad\quad + \max \left \{ \log |w| - \log |u| - \log C_{\ref*{ziegler-const-2}}^{(m)}, \log |u| + \log\left( \max_{1 \leq j \leq m}|\lambda_j| \right) + \log m - \log |w| \right\}.
\end{align*}
\end{proof}

We now state and prove a lemma to obtain bounds on $z_i$ in terms of $n_1$. 
\begin{lemma}\label{lem:lemma-2}
There exist constants \newconstantlem{bound}$c_{\ref*{bound}}$,  \newconstantlem{z1}$c_{\ref*{z1}}$, and \newconstantlem{zieg-lem}$c_{\ref*{zieg-lem}}$ such that the following statements hold.
\begin{enumerate}
    \item $|\lambda_1U_{n_1}+\ldots+\lambda_kU_{n_k}|<c_{\ref*{bound}}|\alpha_1|^{n_1}$.
    \item If $n_1>c_{\ref*{z1}}$, then we have $$z_i<\frac{2\log|\alpha_1|}{\log p_i}n_1$$ for $i=1,\dots,s$.
    \item If $(\lambda_1, \dots, \lambda_k)$ admits dominance, then $$|\lambda_1U_{n_1} + \ldots + \lambda_kU_{n_k}|>c_{\ref*{zieg-lem}}|\alpha|^{n_1}.$$
\end{enumerate}
\end{lemma}
\begin{proof}
First, we prove part (1). We apply the triangle inequality and use $|\alpha|>|\alpha_2|\geq\ldots\geq |\alpha_d|$ to compute 
\begin{align*}
|\lambda_nU_n|
    &=  \left|\lambda_n \left(u\alpha^n+\sum_{j=2}^d u_j\alpha_j^n\right)\right|, \\
    &<|\lambda_n||\alpha|^n\sum_{j=1}^d|u_j|,
\end{align*}
so that \begin{align*}
|\lambda_1U_{n_1}+\ldots+\lambda_kU_{n_k}|
    &\leq|\lambda_1U_{n_1}|+\ldots+|\lambda_kU_{n_k}| \\
    &<    |\lambda_1||\alpha|^{n_1}\sum_{j=1}^d|u_j|+\ldots+|\lambda_k||\alpha|^{n_k}\sum_{j=1}^d|u_j|.
\end{align*}
Since we've assumed that $n_1\geq n_2 \geq\ldots\geq n_k$, we may re-write the above inequality as 
\begin{align*}
    |\lambda_1U_{n_1}+\ldots+\lambda_kU_{n_k}|&<c_{\ref*{bound}}|\alpha|^{n_1},
\end{align*}
where $$c_{\ref*{bound}}=k(\max_{1 \leq j \leq k}|\lambda_j|)\sum_{j=1}^d|u_j|.$$

Next we prove (2). So, using (1), we recall that for all $i$ from $1 \leq i \leq s$, 
\begin{align*}
    |w|p_i^{z_i}\leq |w|p_1^{z_1}\cdots p_s^{z_s}&=\lambda_1U_{n_1}+\ldots+\lambda_kU_{n_k} <c_{\ref*{bound}}|\alpha|^{n_1}.
\end{align*}
Taking real-valued logarithms in the above inequality, we have \begin{align}\label{eq:lemma-2-first-inequality}
    z_i\log p_i&\leq n_1\log|\alpha|\left(1+\frac{\log_*(c_{\ref*{bound}}/|w|)}{n_1\log|\alpha|}\right).
\end{align}
If we assume that $n_1> c_{\ref*{z1}} := \log_{*} (c_{\ref*{bound}} / |w|)/\log |\alpha|$, then, from inequality \eqref{eq:lemma-2-first-inequality} we get \begin{align*}
    z_i&<\frac{2\log|\alpha|}{\log p_i}n_1,
\end{align*}

For part (3), we defer the details to \cite[Proposition 2]{ziegler2019effective}.
\end{proof}

\subsection{Induction Argument Bounding \texorpdfstring{$n_1 - n_m$}{n1 - nm}}\label{sec:generalized-bounding-difference}

Now, we proceed by induction on $m$ where $2 \leq m \leq k$ to obtain a bound on $n_1 - n_m$.
In order to obtain an upper bound on our linear form that we will construct shortly, we first prove a lemma from Ziegler \cite{ziegler2019effective}. 
\begin{lemma}[Ziegler, \cite{ziegler2019effective}]\label{lem:generalized-lemma-1}
Under the assumptions above and assuming that $n \geq 3$, there exists a constant \newconstant{lemma-1-const}$C_{\ref*{lemma-1-const}}$ such that \begin{align}\label{ineq:lemma-1-inequality}
    |U_n-u\alpha^n|=\left|\sum_{j=2}^d u_j\alpha_j^n\right|&<C_{\ref*{lemma-1-const}}|\alpha_2|^n.
\end{align}
\end{lemma}
Since $|\alpha_2| \geq \ldots \geq |\alpha_k|$, we can bound the sum. Set $C_{\ref*{lemma-1-const}} = (k-1)u_{\max}$ where we define $u_{\max} := \max_{1\leq i\leq k}\{|u_i|\}$. \\
\begin{proposition} Under the same hypotheses as in Theorem \ref{thm:main-thm-generalized}, there exist computable constants {$N_1,\dots,N_k$} such that 
    $$n_1 - n_m \leq N_m (\log n_1)^{m-1}$$ 
for $2 \leq m \leq k$.

\end{proposition}
\begin{proof}

We begin by re-writing equation~\eqref{eq:main-equation-generalized} to collect the ``large'' terms on the left-hand side, and then bound those terms.
\begin{align}\label{eq:induct_lambda}
\sum_{j=1}^{m - 1} \left( u\lambda_j\alpha^{n_j} \right) - wp_1^{z_1} \cdots p_s^{z_s}
    &=  -\lambda_mU_{n_m} - \sum_{j=1}^{m - 1} \lambda_j\left( u_2\alpha_{2}^{n_j} +\ldots + u_{d}\alpha_{d}^{n_j} \right), \nonumber\\
\left\lvert \sum_{j=1}^{m - 1} \left( u\lambda_j\alpha^{n_j} \right) - wp_1^{z_1} \cdots p_s^{z_s} \right\rvert
    &=      \left\lvert \lambda_mU_{n_m} + \sum_{j=1}^{m - 1} \lambda_j\left( u_2\alpha_{2}^{n_j} +\ldots + u_{d}\alpha_{d}^{n_j} \right) \right\rvert, \nonumber \\
    &\leq   \left( \max_{1 \leq j \leq m}|\lambda_j| \right)|U_{n_m}| + \sum_{j=1}^{m - 1} |\lambda_j|\left( |u_2\alpha_{2}^{n_j}| +\ldots + |u_{d}\alpha_{d}^{n_j}|  \right), \nonumber\\
    &<     \left( \max_{1 \leq j \leq m}|\lambda_j| \right)\left(\left(|u| + C_1 \right)|\alpha|^{n_m} + |u_{\max}| \sum_{j=1}^{m - 1} \left( |\alpha_{2}^{n_j}| +\ldots + |\alpha_{d}^{n_j}| \right)\right).
\end{align}
Denote $\Lambda_m$ as 
$$\Lambda_m = 1 - |w|p_1^{z_1} \cdots p_s^{z_s} \left\lvert \sum_{j=1}^{m - 1} \left( u\lambda_j\alpha^{n_j} \right)\right\rvert^{-1}.$$

Now, we consider two cases. First, assume $|\alpha_2| < 1$. Then, recalling that $\alpha$ is the dominant root, from equation~\eqref{eq:induct_lambda} we have
\newconstant{diff-const-1}
\begin{align}\label{eq:approx_lambda}
\left\lvert \sum_{j=1}^{m - 1} \left( u\lambda_j\alpha^{n_j} \right) - wp_1^{z_1} \cdots p_s^{z_s} \right\rvert
    &<  C_{\ref*{diff-const-1}}^{(m)} |\alpha|^{n_m},
\end{align}
where
$$C_{\ref*{diff-const-1}}^{(m)} = \left( \max_{1 \leq j \leq m}|\lambda_j| \right) \left( |u| + C_1 + (d - 1)(m - 1)|u_{\max}|\right).$$ 
Next, we divide both sides of equation~\eqref{eq:approx_lambda} by $\left\lvert \sum_{j=1}^{m - 1} \left( u\lambda_j\alpha^{n_j} \right) \right\rvert$ and apply Proposition \ref{prop:proposition-1} to get the following upper-bound on $|\Lambda_m|$,
\begin{align*}
|\Lambda_m|
    &<  \frac{C_{\ref*{diff-const-1}}^{(m)} |\alpha|^{n_m}}{|u| \left\lvert \sum_{j=1}^{m - 1} \lambda_j\alpha^{n_j} \right\rvert} \\
    &< \frac{C_{\ref*{diff-const-1}}^{(m)}}{C_{\ref*{ziegler-const-2}}^{(m)}|u||\alpha|^{n_1 - n_m}}.
\end{align*}

Next, we assume $|\alpha_2| \geq 1$. Similar to before, we divide both sides of by $\left\lvert \sum_{j=1}^{m - 1} \left( u\lambda_j\alpha^{n_j} \right) \right\rvert$ and use that $\alpha$ is the dominant root to get an upper bound on $|\Lambda_m|$,
\newconstant{diff-const-2}
\begin{align*}
|\Lambda_m|
    &<  \frac{d|\alpha|^{n_m}}{|u| \left|\sum_{j=1}^{m-1} \lambda_j\alpha^{n_j}\right|} + \frac{\left( \max_{1 \leq j \leq m}|\lambda_j| \right) |u_{\max}| \sum_{j=1}^{m - 1} \left( |\alpha_{2}^{n_j}| +\ldots + |\alpha_{d}^{n_j}| \right)}{|u| \left|\sum_{j=1}^{m-1} \lambda_j\alpha^{n_j}\right|} \\
    &<  \frac{d|\alpha|^{n_m}}{C_{\ref*{ziegler-const-2}}^{(m)}|u||\alpha|^{n_1}} + \frac{(m - 1)(d - 1)\left( \max_{1 \leq j \leq m}|\lambda_j| \right)|u_{\max}||\alpha_2|^{n_1}}{C_{\ref*{ziegler-const-2}}^{(m)}|u||\alpha|^{n_1}} \\
    &<  \frac{d}{C_{\ref*{ziegler-const-2}}^{(m)}|u|} \left( \frac{|\alpha_2|}{|\alpha|} \right)^{n_1 - n_m} + \frac{(m - 1)(d - 1)\left( \max_{1 \leq j \leq m}|\lambda_j| \right)|u_{\max}|}{C_{\ref*{ziegler-const-2}}^{(m)}|u|} \left( \frac{|\alpha_2|}{|\alpha|} \right)^{n_1 - n_m} \\
    &=  C_{\ref*{diff-const-2}}^{(m)} \left( \frac{|\alpha_2|}{|\alpha|} \right)^{n_1 - n_m},
\end{align*}
where $C_{\ref*{diff-const-2}}^{(m)} = \left( d + (m - 1)(d - 1)\left( \max_{1 \leq j \leq m}|\lambda_j| \right)|u_{\max}|\right)/(C_{\ref*{ziegler-const-2}}^{(m)}|u|)$. Thus, for any $|\alpha_2|$, we have 
\newconstant{diff-const-3}
\begin{align}\label{eq:generalized-induction-step-upperbound}
|\Lambda_m|
    &< \frac{C_{\ref*{diff-const-3}}^{(m)}}{\min\left\{ \frac{|\alpha|}{|\alpha_2|}, |\alpha|\right\}^{n_1 - n_m}},
\end{align}
where $C_{\ref*{diff-const-3}}^{(m)} = \max\{\frac{C_{\ref*{diff-const-1}}^{(m)}}{C_{\ref*{ziegler-const-2}}^{(m)}|u|},C_{\ref*{diff-const-2}}^{(m)}\}$.

Since we've considered much of the computations needed for the induction in the previous section, we proceed showing the induction step (i.e. bounding the difference $|n_1 - n_m|$) as the base case follows a very similar calculation.

First, we apply Theorem \ref{thm:matveev} to obtain a lower bound of $\Lambda_m$. In order to do so, we require that $\Lambda_m \neq 0$. If it were the case that $\Lambda_m = 0$, then, we have
\begin{align}\label{eq:lambda_m equals zero}
|\alpha|^{n_1}p_1^{-z_1}\cdots p_s^{-z_s}
    &=  |w| |u|^{-1} |\lambda_1 + \lambda_2\alpha^{n_2 - n_1} +\ldots + \lambda_{m-1}\alpha^{n_{m-1} - n_1}|^{-1}.
\end{align}
Examining the heights of equation \ref{eq:lambda_m equals zero} and using Lemma \ref{lem:lemma-2} (iii), we have
\begin{align*}
h\left( |w| |u|^{-1} |\lambda_1 + \lambda_2\alpha^{n_2 - n_1} +\ldots + \lambda_{m-1}\alpha^{n_{m-1} - n_1}|^{-1} \right)
    &=  h\left( |\alpha|^{n_1} p_1^{-z_1}\cdots p_s^{-z_s} \right), \\
    &=  n_1 h(|\alpha|) + \sum_{j=1}^{s} z_j \log p_j, \\
    &>  n_1 (h(|\alpha|) + \log |\alpha|) + \log c_{\ref*{zieg-lem}}.
\end{align*}
Recalling the notation of $\gamma_{m-1}$ from Lemma \ref{lem:mazumdar-lemma3.7},
\begin{align*}
\gamma_{m-1}
    &=  |w| |u|^{-1} |\lambda_1 + \lambda_2\alpha^{n_2 - n_1} +\ldots + \lambda_{m-1}\alpha^{n_{m-1} - n_1}|^{-1}.
\end{align*}
Then, using the proof of Lemma \ref{lem:mazumdar-lemma3.7} and the induction hypothesis, we find the following upper-bound on the height of $\gamma_{m-1}$,
\begin{align*}
\begin{split}
h\left( \gamma_{m-1} \right)
    &\leq   \log |w| + h(|u|) + (m-1) \log \left( \max_{1 \leq j \leq m-1}|\lambda_j| \right) \\
    &\quad\quad + \left( \sum_{j=2}^{m-1} |n_j - n_1| \right)  h(|\alpha|) + \log (m - 1), \\
    &\leq   \log |w| + h(|u|) + (m-1) \log \left( \max_{1 \leq j \leq m-1}|\lambda_j| \right) \\
    &\quad\quad + \left( N_2\log n_1 + N_3 (\log n_1)^2 + \ldots + N_{m - 1} (\log n_1)^{m-2} \right)  h(|\alpha|) + \log (m-1), \\
    &\leq   \log |w| + h(|u|) + (m-1)  \log \left( \max_{1 \leq j \leq m-1}|\lambda_j| \right) \\
    &\quad\quad + \left( N_2 + N_3 + \ldots + N_{m - 1} \right) h(|\alpha|) (\log n_1)^{m-2} + \log (m - 1). \\
\end{split}
\end{align*}
Next, comparing inequalities  applying Lemma \ref{lem:not_log}, we have the following bound on $n_1$,
\begin{align*}
n_1
    &\leq   \max\left\{ 2^h(u^{1/(m-2)}+v^{1/(m-2)}\log ((m-2)^{m-2} v))^{m-2},2^{m-2} (u^{1/(m-2)}+2 e^2)^{m-2} \right\},
\end{align*}
where we define
\begin{align*}
u
    &=  \frac{\log |w| + h(|u|) + (m - 1)  \log \left( \max_{1 \leq j \leq m-1}|\lambda_j| \right) + \log (m - 1) - \log c_5}{h(|\alpha|) + \log |\alpha|}, \\
v
    &=  \frac{(N_2 + N_3 + \ldots + N_{m - 1}) h(|\alpha|)}{h(|\alpha|) + \log |\alpha|}.
\end{align*}
Thus, if $\Lambda_m = 0$, we may find an upper-bound on $n_1$.

Now, we proceed assuming $\Lambda_m \neq 0$. Let $D = [\Q(\alpha, \dots, \alpha_d) : \Q]$.
Let $\gamma_{1, i} = p_i$ and $b_{1, i} = z_i$, and let $\gamma_2 = \alpha$ and $b_2 = -n_1$. Similarly, we set $\gamma_3 = |w| |u|^{-1} |\lambda_1 + \lambda_2\alpha^{n_2 - n_1} +\ldots + \lambda_{m-1}\alpha^{n_{m-1} - n_1}|^{-1}$ and $b_3 = 1$. Since $h(p_i) = \log p_i$, we choose $A_{1, i} = D \log p_i$. We choose $A_2 = \max\{Dh(|\alpha|), \log |\alpha|, 0.16\}$ and $A_3 = A_3(m - 1)$ as defined in Lemma \ref{lem:mazumdar-lemma3.7}. Finally, we let $B = d_1n_1 \geq \max\{|z_1|, \dots, |z_s|, |n_1|, 1\}$, where $d_1 = \frac{2\log |\alpha|}{\log 2} n_1$. Next, we get a bound on $A_3$ by using Lemma \ref{lem:mazumdar-lemma3.7} and our induction hypothesis.
\begin{align}\label{ineq:generalized-induction-step-hypothesis}
    A_3     &=  DC_{\ref*{maz-lemma-const-1}}^{(m-1)} + D\left( \sum_{j=2}^{m - 1} |n_1 - n_j|\right) h(|\alpha|) \nonumber \\
            &<  DC_{\ref*{maz-lemma-const-1}}^{(m-1)} + D\left( \sum_{j=2}^{m-1} N_j (\log n_1)^{j - 1} \right) h(|\alpha|) \nonumber \\
            &=  DC_{\ref*{maz-lemma-const-1}}^{(m-1)} + D\left(N_2\log n_1 + N_3 (\log n_1)^2 + \ldots + N_{m - 1} (\log n_1)^{m-2}\right) h(|\alpha|) \nonumber \\
            &<  \left( DC_{\ref*{maz-lemma-const-1}}^{(m-1)} + D\left(N_2 + N_3 + \ldots + N_{m-1} \right) h(|\alpha|) \right) (\log n_1)^{m - 2}.
\end{align}
Then applying Theorem \ref{thm:matveev}, we have \newconstant{diff-const-4}
\begin{align*}
\log |\Lambda_m|    
    &>  -1.4 \cdot 30^{s+5}(s+2)^{4.5} D^2 (1 + \log D) (1 + \log d_1n_1) (D \log p_1) \cdots (D \log p_s) A_2 A_3, \\
    &>  -C_M^{(s + 2)} (1 + \log d_1) (\log n_1) (\log p_1) \cdots (\log p_s) A_2 A_3, \nonumber \\
    &>  -C_{\ref*{diff-const-4}}^{(m-1)} (\log n_1)^{m - 1}, \nonumber
\end{align*}
where we used inequality (\ref{ineq:generalized-induction-step-hypothesis}) and let 

$$C_{\ref*{diff-const-4}}^{(m-1)} = C_M^{(s + 2)} (1 + \log d_1) \left(\prod_{i=1}^s\log p_i\right) A_2  \left( DC_{\ref*{maz-lemma-const-1}}^{(m-1)} + D\left(\sum_{i=2}^{m-1} N_i\right) \cdot h(|\alpha|) \right).$$ 
Taking logarithms of inequality \eqref{eq:generalized-induction-step-upperbound}, we obtain
\begin{align*}
\log |\Lambda_m|  &<  \log C_{\ref*{diff-const-3}}^{(m)} - (n_1 - n_m) \log \min\left\{ \frac{|\alpha|}{|\alpha_2|}, |\alpha| \right\}.
\end{align*}
Thus, comparing the above inequalities, we have
\begin{align*}
-C_{\ref*{diff-const-4}}^{(m-1)} \cdot (\log n_1)^{m - 1}
    &<  \log C_{\ref*{diff-const-3}}^{(m)} - (n_1 - n_m) \log \min\left\{ \frac{|\alpha|}{|\alpha_2|}, |\alpha|\right\}, \\
n_1 - n_m
    &<  \frac{\log C_{\ref*{diff-const-3}}^{(m)} + C_{\ref*{diff-const-4}}^{(m-1)} \cdot (\log n_1)^{m - 1}}{\log \min\left\{ \frac{|\alpha|}{|\alpha_2|}, |\alpha|\right\}}, \\
    &<  \frac{\log C_{\ref*{diff-const-3}}^{(m)} + C_{\ref*{diff-const-4}}^{(m-1)}}{\log \min\left\{ \frac{|\alpha|}{|\alpha_2|}, |\alpha|\right\}} (\log n_1)^{m - 1}, \\
    &<  N_m (\log n_1)^{m - 1},
\end{align*}
where $N_m = \left(C_{\ref*{diff-const-4}}^{(m)} + \log C_{\ref*{diff-const-3}}^{(m)}\right)/\left(\log\min\left\{\frac{|\alpha|}{|\alpha_2|}, |\alpha|\right\}\right)$.
\end{proof}

\subsection{Bounding \texorpdfstring{$n_1$}{n1}}\label{sec:generalized-bounding-n1}
In this section, we find a bound on $n_1$ in terms of $n_m - n_1$ for $2 \leq m \leq k$. First, we begin by re-writing equation~\eqref{eq:main-equation-generalized} to collect all of the ``large'' terms on the left-hand side,
\begin{align*}
\lambda_1U_{n_1} +\ldots + \lambda_kU_{n_k}  
    &=  wp_1^{z_1}\cdots p_s^{z_s} \\
\sum_{i=1}^d u_i\lambda_1\alpha_i^{n_1} +\ldots + \sum_{i=1}^d u_i\lambda_k\alpha_i^{n_k}
    &=  wp_1^{z_1} \cdots p_s^{z_s} \\
\sum_{j=1}^k u\lambda_j\alpha^{n_j} - wp_1^{z_1} \cdots p_s^{z_s}
    &=  -\sum_{i=2}^d u_i\lambda_1\alpha_i^{n_1} - \ldots - \sum_{i=2}^d u_i\lambda_k\alpha_i^{n_k}.
\end{align*}
Taking absolute values, we obtain
\begin{align*}
\left\lvert u\alpha^{n_1} \sum_{j=1}^k \lambda_j\alpha^{n_j - n_1} - wp_1^{z_1} \cdots p_s^{z_s} \right\rvert
    &=  \left\lvert \sum_{i=2}^d u_i\lambda_1\alpha_i^{n_1} +\ldots + \sum_{i=2}^d u_i\lambda_k\alpha_i^{n_k} \right\rvert \\
    &=  \left\lvert \sum_{j=1}^k \lambda_j\left(u_2\alpha_2^{n_j} +\ldots + u_d\alpha_d^{n_j} \right) \right\rvert.
\end{align*}
Denote $\Lambda$ as follows,
$$\Lambda = 1 - |w|p_1^{z_1} \cdots p_s^{z_s} |u|^{-1}|\alpha|^{-n_1}|\lambda_1 + \lambda_2\alpha^{n_2 - n_1} +\ldots + \lambda_{k}\alpha^{n_{k} - n_1}|^{-1}.$$

Now, we find an upper bound for $\Lambda$ by considering two different cases. If $|\alpha_2| < 1$, then we have
\begin{align*}
\left\lvert \sum_{j=1}^k \lambda_j\left(u_2\alpha_2^{n_j} +\ldots + u_d\alpha_d^{n_j} \right) \right\rvert
    &<      \left\lvert \sum_{j=1}^k \lambda_j\left(u_2 +\ldots + u_d \right) \right\rvert \\
    &\leq   k(d - 1)(\max_{1 \leq j \leq k}|\lambda_j|)u_{\max}.
\end{align*}
Using our definition of $\Lambda$ and dividing the above equation by $\left\lvert \sum_{j=1}^k u\lambda_j\alpha^{n_j} \right\rvert$, along with Proposition \ref{prop:proposition-1}, gives \newconstant{n1-const}
\begin{align*}
|\Lambda|
    &<      \frac{k(d - 1)(\max_{1 \leq j \leq k}|\lambda_j|)u_{\max}}{|u| \left\lvert \sum_{j=1}^k \lambda_j\alpha^{n_j} \right\rvert} \\
    &\leq   \frac{k(\max_{1 \leq j \leq k}|\lambda_j|)u_{\max}}{C_{\ref*{ziegler-const-2}}^{(m)}|u| |\alpha|^{n_1}} \\
    &=      \frac{C_{\ref*{n1-const}}}{|\alpha|^{n_1}},
\end{align*}
where $C_{\ref*{n1-const}} = \left(k(d-1)(\max_{1 \leq j \leq k}|\lambda_j|)u_{\max}\right)/(C_{\ref*{ziegler-const-2}}^{(m)}|u|)$. 

If $|\alpha_2| \geq 1$, then by a similar process, 
\begin{align*}
|\Lambda|
    &<      \frac{\left\lvert \sum_{j=1}^k \lambda_j\left(u_2\alpha_2^{n_j} +\ldots + u_d\alpha_d^{n_j} \right) \right\rvert}{|u| \left\lvert \sum_{j=1}^k \lambda_j\alpha^{n_j} \right\rvert} \\
    &<      \frac{\left\lvert \sum_{j=1}^k u_2\lambda_j\alpha_2^{n_j} \right\rvert +\ldots + \left\lvert \sum_{j=1}^k u_d\lambda_j\alpha_d^{n_j} \right\rvert}{C_{\ref*{ziegler-const-2}}^{(m)}|u||\alpha|^{n_1}} \\
    &\leq   \frac{k\left( \max_{1 \leq j \leq k}|\lambda_j| \right)|u_2|}{C_{\ref*{ziegler-const-2}}^{(m)}|u|} \left( \frac{|\alpha_2|}{|\alpha|} \right)^{n_1} + \ldots + \frac{k\left( \max_{1 \leq j \leq m}|\lambda_j| \right)|u_d|}{C_{\ref*{ziegler-const-2}}^{(m)}|u|} \left( \frac{|\alpha_d|}{|\alpha|} \right)^{n_1} \\
    &<      \frac{k(d - 1)\left( \max_{1 \leq j \leq k}|\lambda_j| \right)u_{\max}}{C_{\ref*{ziegler-const-2}}^{(m)}|u|} \left( \frac{|\alpha_2|}{|\alpha|} \right)^{n_1} \\
    &=      C_{\ref*{n1-const}} \left( \frac{|\alpha_2|}{|\alpha|} \right)^{n_1}.
\end{align*}
Thus, for any $|\alpha_2|$, we have the following,
\begin{align}\label{eq:generalized-bounding-n1-upperb}
|\Lambda|
    &<  \frac{C_{\ref*{n1-const}}}{\min \left\{\frac{|\alpha|}{|\alpha_2|}, \alpha\right\}^{n_1}}.
\end{align}

We now use Theorem \ref{thm:matveev} to find a lower bound for $|\Lambda|$. As before, we first handle the $\Lambda = 0$ case. Similar to Section \ref{sec:generalized-bounding-difference}, we recall that $\gamma_k$ is defined as follows
\begin{align*}
\gamma_k
    &=  |w| |u|^{-1} |\lambda_1 + \lambda_2\alpha^{n_2 - n_1} +\ldots + \lambda_{k}\alpha^{n_{k} - n_1}|^{-1} 
\end{align*}
Next, we find an upper-bound on the height of $\gamma_k$ by a very similar process,
\begin{align*}
\begin{split}
h(\gamma_k)
    &\leq   \log |w| + h(|u|) + k  \log \left( \max_{1 \leq j \leq k}|\lambda_j| \right) + \log k \\
    &\quad\quad +  (N_2 + N_3 + \ldots + N_{k}) h(|\alpha|) (\log n_1)^{k-1}.
\end{split}
\end{align*}
Using Lemma \ref{lem:lemma-2}, we have the following upper-bound on $n_1$,
\begin{align*}
n_1
    &\leq   \max \left\{ 2^{k-1}(u^{1/(k-1)}+v^{1/(k-1)}\log ((k-1)^{k-1} v))^{k-1},2^{k-1}(u^{1/(k-1)}+2e^2)^{k-1} \right\},
\end{align*}
where we define
\begin{align*}
u
    &=  \frac{\log |w| + h(|u|) + k  \log \left( \max_{1 \leq j \leq k}|\lambda_j| \right) + \log k - \log c_5}{h(|\alpha|) + \log |\alpha|} \\
v
    &=  \frac{(N_2 + N_3 + \ldots + N_{k}) h(|\alpha|)}{h(|\alpha|) + \log |\alpha|}
\end{align*}

Now, we assume that $\Lambda \neq 0$. Let $\gamma_{1, i} = p_i$ and $b_{1, i} = z_i$. Similar to the previous section, we let $\gamma_2 = \alpha$, $b_2 = -n_1$, $\gamma_3 = |w| |u|^{-1}\left\lvert \lambda_1 + \lambda_2\alpha^{n_2 - n_1} + \ldots + \lambda_{k}\alpha^{n_{k} - n_1} \right\rvert^{-1}$, and $b_3 = 1$. Denote $D = [\Q(\alpha, \dots, \alpha_d) : \Q]$. We choose $A_{1, i} = D \log p_i$ for $1 \leq i \leq s$. Similarly, we let $A_2 = \max\{Dh(\alpha), \log |\alpha|, 0.16\}$. For $A_3$, we bound it as follows,
\begin{align*}
A_3     
    &=  DC_{\ref*{maz-lemma-const-1}}^{(k)} + D\left( \sum_{j=2}^{k} |n_1 - n_j|\right) h(|\alpha|) \\
    &<  DC_{\ref*{maz-lemma-const-1}}^{(k)} + D\left( \sum_{j=2}^{k} N_j (\log n_1)^{j - 1} \right) h(|\alpha|) \\
    &=  DC_{\ref*{maz-lemma-const-1}}^{(k)} + D\left(N_2\log n_1 + N_3 (\log n_1)^2 + \ldots + N_{k} (\log n_1)^{k-1}\right) h(|\alpha|) \\
    &<  \left( D C_{\ref*{maz-lemma-const-1}}^{(k)} + D\left(N_2 + N_3 + \ldots + N_{k} \right) h(|\alpha|) \right) (\log n_1)^{k - 1}.
\end{align*}
Thus,
\begin{align}\label{eq:generalized-bounding-n1-lowerb}
\log |\Lambda|
    &>  -1.4 \cdot 30^{s+5} (s+2)^{4.5} D^2  (1 + \log D) (1 + \log d_1n_1) (D \log p_1) \cdots (D \log p_s) A_2 A_3 \nonumber \\
    &>  -C_M(s + 2) (1 + \log d_1) (\log n_1) (\log p_1) \cdots (\log p_s) A_2 A_3 \nonumber \\
    &>  -C_{\ref*{diff-const-4}}^{(k)} (\log n_1)^{k}.
\end{align}
Taking logarithms of inequality \eqref{eq:generalized-bounding-n1-upperb},
\begin{align}\label{eq:generalized-bounding-n1-upperb-log}
\log |\Lambda|  &<  \log C_{\ref*{n1-const}} - n_1 \log \min\left\{ \frac{|\alpha|}{|\alpha_2|}, |\alpha| \right\}.
\end{align}
Then, comparing inequalities \eqref{eq:generalized-bounding-n1-lowerb} and \eqref{eq:generalized-bounding-n1-upperb-log} and recalling that $n_1 \geq 3$, we obtain 
\begin{align*}
-C_{\ref*{diff-const-4}}^{(k)} (\log n_1)^{k}    
    &<  \log C_{\ref*{n1-const}} - n_1 \log \min\left\{ \frac{|\alpha|}{|\alpha_2|}, |\alpha| \right\}
\end{align*}
and thus
\begin{align*}
n_1
    &<  N_{\max} (\log n_1)^{k}
\end{align*} where $N_{\max} = \left(\log C_{\ref*{n1-const}} + C_{\ref*{diff-const-4}}^{(k)}\right)/\left(\log\min\left\{\frac{|\alpha|}{|\alpha_2|},|\alpha|\right\}\right)$. \\

Finally, using Lemma \ref{lem:not_log}, we have
\begin{align*}
n_1     &<  2^{k} \max\left\{N_{\max} \left( \log \left(k^{k} N_{\max} \right) \right)^{k}, (2e^2)^{k}\right\}.
\end{align*}
Since $n_1 > \ldots > n_k$, and we can bound each $z_i$ in terms of $n_1$ by Lemma \ref{lem:lemma-2}, this concludes the proof of Theorem \ref{thm:main-thm-generalized}.

\newpage

\section{Acknowledgements}
The authors are grateful for the support and funding received from SMALL REU 2021 and from  NSF Grant DMS1947438. They would also like to thank Volker Ziegler and Ingrid Vukusic for their many helpful comments on the paper.

\section*{Appendix}\label{sec:appendix-constants}

\begin{table}[h!]
\caption {Constants} \label{tab:constants} 
\renewcommand{\arraystretch}{1.5}
\begin{center}
\begin{tabular}{|l|l|}
\hline
      \multicolumn{2}{|l|}{\textbf{Theorem~\ref{thm:matveev} (Matveev's theorem)}} \tabularnewline
      \hline $C_M(x)$ & $1.4 \cdot 30^{x + 2}  x^{4.5} D^{x+2} (1 + \log D)$ \tabularnewline
      \hline 
      \multicolumn{2}{|l|}{\textbf{Lemma \ref{lem:mazumdar-lemma3.7}}} \tabularnewline
      \hline $C_{\ref*{maz-lemma-const-1}}^{(m)}$ & $\log |w| + h(|u|) + m \log (\max_{1 \leq j \leq m}|\lambda_j|) + \log m$
      \tabularnewline
      & $\quad+ \max\{\log |w| - \log |u| - \log C_{\ref*{ziegler-const-2}}^{(m)},\, \log |u| + \log (\max_{1 \leq j \leq m}|\lambda_j|) + \log m - \log |w|\}$
      \tabularnewline
      \hline
      \multicolumn{2}{|l|}{\textbf{Lemma~\ref{lem:generalized-lemma-1}}} 
      \tabularnewline
      \hline $C_{\ref*{lemma-1-const}}$ & $(k-1)u_{\max}$, where $u_{\max}=\max_{1\leq i\leq k}|u_i|$
      \tabularnewline
      \hline
      \multicolumn{2}{|l|}{\textbf{Induction Argument Bounding $n_1-n_m$}} \tabularnewline
      \hline
      $C_{\ref*{diff-const-1}}^{(m)}$ & $\left( \max_{1 \leq j \leq m}|\lambda_j| \right) \left( |u| + C_1 + (d - 1)(m - 1)|u_{\max}|\right)$ 
      \tabularnewline
      \hline
      $C_{\ref*{diff-const-2}}^{(m)}$ & $\left( d + (m - 1)(d - 1)\left( \max_{1 \leq j \leq m}|\lambda_j| \right)|u_{\max}|\right)/(C_{\ref*{ziegler-const-2}}^{(m)}|u|)$ 
      \tabularnewline
      \hline
      $C_{\ref*{diff-const-3}}^{(m)}$ & $\max\{\frac{C_{\ref*{diff-const-1}}^{(m)}}{C_{\ref*{ziegler-const-2}}|u|},C_{\ref*{diff-const-2}}^{(m)}\}$
     \tabularnewline
     \hline
     $C_{\ref*{diff-const-4}}^{(m-1)}$ & $C_M^{(s + 2)} (1 + \log d_1) (\prod_{i=1}^s\log p_i) A_2  ( DC_{\ref*{maz-lemma-const-1}}^{(m-1)} + D(\sum_{i=2}^{m-1} N_i) \cdot h(|\alpha|) )$,
     \tabularnewline
     & $\quad$ where $A_2=\max\{Dh(|\alpha)|,\log|\alpha|,0.16\}$
     \tabularnewline
     \hline
     $N_m$ & $(C_{\ref*{diff-const-4}}^{(m)} + \log C_{\ref*{diff-const-3}}^{(m)})/(\log \min \{\frac{|\alpha|}{|\alpha_2|}, |\alpha|\})$
     \tabularnewline
     \hline
     \multicolumn{2}{|l|}{\textbf{Bounding $n_1$}} \tabularnewline
   \hline
   $C_{\ref*{n1-const}}$ & $(k(d - 1)(\max_{1 \leq j \leq k}|\lambda_j|)u_{\max})/(C_{\ref*{ziegler-const-2}}^{(m)}|u|)$ 
   \tabularnewline
      \hline
      $N_{\max}$ & $(C_{\ref*{n1-const}} + C_{\ref*{diff-const-4}}^{(k)})/(\log \min \{ \frac{|\alpha|}{|\alpha_2|}, |\alpha| \})$
      \tabularnewline
      \hline
\end{tabular}
\end{center}
\end{table}

\bigskip

\bibliographystyle{acm}
\bibliography{main_arxiv}

\end{document}